\documentclass[11pt,leqno]{article}  
\usepackage{fullpage}
\usepackage{amsmath}
\usepackage{amsfonts}
\usepackage{amssymb}

\usepackage{enumerate}

\newcommand{\IGNORE}[1]{}
\newcommand{\ds}{\displaystyle}
\newcommand{\myfloor}[1]{\lfloor {#1} \rfloor}

\newcommand\tn[1]{\textnormal{#1}}

\newcommand{\symmdiff}{\bigtriangleup}

\begin{document}


{
{
\title{ On Eulerian orientations of even-degree hypercubes}

\IGNORE{
\author[uw]{Maxwell Levit
}

\author[iisc]{L.~Sunil Chandran
}

\author[uw]{Joseph Cheriyan\corref{cor1}\fnref{fn1}
}
\ead{jcheriyan@uwaterloo.ca}
\ead[url]{http://www.math.uwaterloo.ca/\~{}jcheriyan}

\fntext[fn1]{This author acknowledges support from the Natural Sciences
\& Engineering Research Council of Canada (NSERC), No.~RGPIN--2014--04351.}

\address[uw]{
C\&O Dept., University of Waterloo, 
Waterloo, ON, Canada N2L~3G1
}

\address[iisc]{
Computer Science and Automation Dept., 
Indian Institute of Science, Bangalore, India
}
end \IGNORE
}

\author{
Maxwell Levit
\thanks {
C\&O Dept., University of Waterloo, 
Waterloo, Canada.
}
\and
L.~Sunil Chandran
\thanks {
Computer Science and Automation Dept., 
Indian Institute of Science, Bangalore, India.
}
\and
Joseph Cheriyan
\thanks {
C\&O Dept., University of Waterloo, 
Waterloo, Canada.
}
\thanks{This author acknowledges support from the Natural Sciences
\& Engineering Research Council of Canada (NSERC), No.~RGPIN--2014--04351.}
}

}

  \date{}
\maketitle

\begin{abstract}
{
It is well known that \textit{every} Eulerian orientation
of an Eulerian $2k$-edge connected (undirected) graph
is strongly $k$-edge connected.
A long-standing goal in the area is to
obtain analogous results for
other types of connectivity,
such as node connectivity and element connectivity.
We show that \textit{every} Eulerian orientation
of the hypercube of degree $2k$
is strongly $k$-node connected.
}
\end{abstract}

\IGNORE{
}

\medskip

\medskip


}


\newtheorem{theorem}{Theorem}
\newtheorem{corollary}[theorem]{Corollary}
\newtheorem{fact}[theorem]{Fact}
\newtheorem{claim}[theorem]{Claim}

\newenvironment{proof}{{\noindent \bf Proof:}}{\hfill\rule{2mm}{2mm}}

\newenvironment{remark}{{\noindent \bf Remark}:}{}



\section{Introduction} \label{sec:intro}

The hypercube $Q_k$ is a $k$-regular graph on $2^k$ nodes
that can be constructed by
labeling the nodes by the $2^k$ subsets of the set $\{1,2,\dots,k\}$
and placing an edge between two nodes whenever
the two node labels (i.e., the two subsets) differ in
a single element.
%
%
{
Hypercubes are ubiquitous, both in mathematics and in the ``real world.''
It can be seen that $Q_k$ is the ``covering graph'' of the family of
all subsets of a $k$-element set,
%
see \cite{Bol},
and hence, properties of hypercubes have universal appeal.
(Recall that the covering graph of
the poset/powerset of a set $U$ has a node for each subset $A$ of
$U$, and the nodes of subsets $A,B$ are adjacent iff $|A \symmdiff B|=1$.)
}
Hypercubes (and their variants) are useful in
computer communication networks, VLSI design, etc., and
there is extensive literature in this area, see \cite{CD90,FKL02,H95,L92,SS88}.

An \textit{orientation} of an (undirected) graph $G=(V,E)$ is a directed~graph
$D=(V,A)$ such that each edge $\{v,w\}\in E$ is replaced
by exactly one of the arcs $(v,w)$ or $(w,v)$.

Orientations of hypercubes have applications in practical domains
such as broadcasting in computer communication networks
and the design of parallel computer architectures.
The connectivity properties of hypercubes and orientations of hypercubes
have been studied, see \cite{CD90,H95,SS88}, and
orientations of hypercubes that achieve the maximum possible node
connectivity are of interest, see \cite[Proposition~9]{H95}.

Our key result states that the optimal node connectivity
among orientations of $Q_{2k}$
can be achieved in a trivial way:
pick any orientation such that
the indegree is equal to the outdegree at every node.


\subsection{Smooth orientations and Eulerian orientations} \label{sec:smooth}

For a node $v$ of a directed graph, we use $d^{in}(v)$
to denote the number of arcs with head~$v$;
similarly, $d^{out}(v)$ denotes the number of arcs with tail~$v$.

An orientation of an (undirected) graph $G$ is called \textit{smooth}
if the absolute value of the difference between the
indegree and the outdegree of every node is at most one, that is,
$\big| d^{in}(v)-d^{out}(v) \big| \le1,\;\forall{v}\in{V(G)}$.
A smooth orientation of an Eulerian graph $G$ is
called an \textit{Eulerian} orientation;
such an orientation satisfies
$d^{in}(v) = d^{out}(v),\;\forall{v}\in{V(G)}$.
Moreover, it can be seen that
for every {Eulerian} orientation,
for every subset of the nodes $W$,
the number of arcs leaving $W$ is equal to
the number of arcs entering $W$, see \cite[Ch.6.1]{Di}.
Therefore, every Eulerian orientation of a
$2k$-edge connected Eulerian graph results in a directed graph
that is $k$-edge connected.
An Eulerian orientation of an Eulerian graph
can be found by orienting the edges of
each connected component according to an Euler tour.


\subsection{Nash-Williams' results and possible extensions} \label{sec:NWtheorems}

A well-known result of Nash-Williams says that
the edges of a $k$-edge connected graph
can be oriented such that the resulting
directed~graph is $\myfloor{\frac{k}{2}}$-edge connected
\cite{N60}, \cite[Ch.9]{BM}.
A long-standing goal in the area is to
extend Nash-Williams' result  to obtain analogous results for
other types of connectivity,
such as node connectivity and element connectivity,
see \cite{F-HC95,KS06,KL08,T89,T2015}.


\subsection{Our results} \label{sec:results}

We show that {every} Eulerian orientation
of the hypercube $Q_{2k}$
is strongly $k$-node connected;
recall that a directed~graph is called \textit{strongly $k$-node connected} if
it has $\ge k+1$ nodes and
the deletion of any set of $\leq (k-1)$ nodes results in a
strongly-connected directed~graph.

Let us mention that there are easy inductive constructions that prove
that there exists a ``good orientation'' for a hypercube of even degree;
we describe one such construction in Fact~\ref{f:exists-goodorient}.
For hypercubes of odd degree, the smoothness condition does not
guarantee ``good orientations;''
for example, there exist smooth orientations of $Q_3$ that are not
strongly connected.


\section{Preliminaries} \label{sec:prelims}

This section has some definitions and preliminary results.
Also, see \cite{Di} for standard definitions and notation.

The hypercube $Q_k$ is the Cartesian product of $k$ copies of $K_2$,
see \cite{wiki-HG}.
There are other constructions of $Q_k$, and we describe three of them.
\begin{enumerate}[(i)]
\item
Label $2^k$ nodes by $k$-bit binary strings, and
place an edge between two nodes whenever
their labels differ in exactly one bit
(i.e., the Hamming distance between the two strings is one).

\item
Label $2^k$ nodes by the $2^k$ subsets of a set with $k$ elements,
and place an edge between two nodes whenever
the two node labels (i.e., the two subsets) differ in a single element.
\item
Take two disjoint hypercubes $Q_{k-1}$, and place an edge
between corresponding pairs of nodes in the 
two copies of $Q_{k-1}$; thus, the edges between the
two copies of $Q_{k-1}$ form a perfect matching.
\end{enumerate}

By a $d$-hypercube we mean a hypercube of degree~$d$.

For a node~set $S$ of a graph $G$, we use $N_G(S)$  to denote the
set of neighbors of $S$, thus,
$N_G(S) = \{ w\in V(G)-S \;:\; \exists v\in S ~\tn{such that}~ \{v,w\}\in E(G) \}$.

\begin{fact} \label{f:exists-goodorient}
For each integer $k\ge1$, there exists an Eulerian orientation of $Q_{2k}$
that is strongly $k$-node connected.
\end{fact}
\begin{proof}
Let $k\ge1$ be an integer.
We sketch an inductive construction that gives a strongly $(k+1)$-node~connected
Eulerian orientation for the hypercube $Q_{2k+2}$.
Observe that any Eulerian orientation of $Q_2$ (the 4-cycle) is
strongly 1-connected.  Assume (by induction) that $Q_{2k}$ has a
strongly $k$-node~connected Eulerian orientation.
View the $(2k+2)$-hypercube as four $2k$-hypercubes (i.e., four copies of $Q_{2k}$)
together with $2^{2k}$ 4-cycles,
where each of these 4-cycles $C_i$
contains a distinct node $i$ of
the first copy of $Q_{2k}$
as well as the image of $i$ in each of the other three copies of $Q_{2k}$.
By the induction hypothesis,
there exists a strongly $k$-node~connected Eulerian orientation for $Q_{2k}$.
Fix such an orientation for each of the four copies of $Q_{2k}$.
Moreover, for each of the 4-cycles $C_i$,
fix any Eulerian orientation of $C_i$.
Let $D$ be the resulting directed~graph (i.e., orientation of $Q_{2k+2}$).
We claim that $D$ is strongly $(k+1)$-node~connected.
To see this, consider any set of nodes $Z$ of size $\le k$.
Suppose that one of the four copies of $Q_{2k}$ contains $Z$;
then it is clear that each of the other three copies of $Q_{2k}$
is strongly connected in $D-Z$, and hence,
(using the $2^{2k}$ oriented 4-cycles of $D$)
it can be seen that $D-Z$ is strongly connected.
Otherwise, each of the four copies of $Q_{2k}$ has $\le k-1$ nodes of $Z$,
hence, the removal of $Z$ from any one of the four copies of $Q_{2k}$
results in a strongly connected directed~graph;
again (using the $2^{2k}$ oriented 4-cycles of $D$),
it can be seen that $D-Z$ is strongly connected.
\end{proof}


\section{Eulerian orientations of $2k$-hypercubes} \label{sec:cube}

This section has our results and proofs.
In this section, we assume that $k$ is a positive integer.

{
\begin{theorem} \label{thm:oriG}
Let $G$ be a $2k$-regular $2k$-node~connected graph
such that for every set of nodes $S$ with $1\leq |S| \leq |V(G)|/2$
we have
$\ds |N_G(S)|  ~>~ \min\{ k^2-1,\; (k-1)(|S|+1) \}$.
Then every Eulerian orientation of $G$ is strongly $k$-node connected. 
\end{theorem}
\begin{proof}
Let $D$ denote an arbitrary Eulerian orientation of $G$.
(In what follows, when we refer to \textit{the orientation} of an edge
of $G$ we mean the corresponding directed edge of $D$.)
By way of contradiction, suppose that $D$ is not strongly $k$-node
connected.  Then there is a node~set $Z$ of size $\leq k-1$ whose
deletion from $D$ results in a directed~graph that has a partition
 $(S, \bar{S})$ of its node~set $V(G)-Z$ such that
 both $S,\bar{S}$ are nonempty and 
 the edges of $G-Z$ in this cut either
 are all oriented from $S$ to $\bar{S}$ or
 are all oriented from $\bar{S}$ to ${S}$.
We fix the notation such that $|S| \leq |\bar{S}|$.
(Now, observe that $|S|$ satisfies the condition stated in the hypothesis.)
Moreover, without loss of generality,
we assume that the edges are oriented from $S$ to $\bar{S}$
(the arguments are similar for the other case).
Observe that $G-Z$ has $\ge |N_G(S)| - |Z|$ edges in the cut $(S, \bar{S})$.
Thus, $D$ has $\ge |N_G(S)| - |Z|$ edges oriented out from $S$
(and into $\bar{S}$).
Consider the cut $(S, \bar{S} \cup{Z})$ of $G$, and
observe that it has $\leq \min\{ k|Z|,\; |S|\,|Z| \}$ edges oriented
into $S$ (and out of $Z$), because
(i)~all such edges are incident to nodes of $Z$ and only $k$
of the $2k$~edges incident to a node $w\in Z$ are oriented out of $w$;
(ii)~each such edge is incident to a node $s\in S$ and a node $w\in Z$
(and each pair $s,w$ contributes at most one such edge).
Thus, the cut $(S, \bar{S} \cup{Z})$ of $G$ has
$\ge |N_G(S)| - |Z| \ge |N_G(S)| - (k-1)$ edges oriented out of $S$ and
$\leq \min\{ k|Z|,\; |S|\,|Z| \} \leq \min\{ k(k-1),\; |S|(k-1) \}$
edges oriented into $S$;
the hypothesis (in the theorem) implies that the former quantity
is greater than the latter quantity.
This is a contradiction:
in an Eulerian orientation of an Eulerian graph,
every cut has the same number of outgoing edges and incoming edges.
\end{proof}
}

In the next subsection we show that hypercubes of even degree
satisfy all the conditions stated in Theorem~\ref{thm:oriG};
this gives our main result.

\subsection{Bounds for the ${2k}$-hypercube} \label{sec:iso}

The main goal of this subsection is to show that the hypercube $Q_{2k}$
satisfies the inequalities stated in Theorem~\ref{thm:oriG}.
Our analysis has two parts depending on the size $m$ of the set $S\subseteq V(Q_{2k})$
(in the statement of Theorem~\ref{thm:oriG});
the first part  (Fact~\ref{f:cube-2}) applies for $1\le m\le 2k+1$ and it follows easily;
the second part (Fact~\ref{f:big-m}) applies for $2k+2\le m\le 2^{2k-1}$ and it follows by
exploiting properties of the hypercube.
In more detail, in the second part, we show that the
minimum of $|N_{Q_{2k}}(S)|$ over all sets $S\subseteq V(Q_{2k})$ of size $m$
(where $2k+2\le m\le 2^{2k-1}$) is $>k^2-1$;
our proof avoids elaborate computations by exploiting structural
properties of hypercubes;
a key point is to focus on a subgraph of the hypercube induced by
the set of binary strings of Hamming weight~$i$ and 
the set of binary strings of Hamming weight~$i-1$
(see Claim~\ref{cl:lower-shadow} in the proof of Fact~\ref{f:big-m}).

We follow the notation of \cite{Bol} and use $b_v(m,Q_{2k})$ to denote
$\min \{ |N_{Q_{2k}}(S)| \;:\; S\subseteq V(Q_{2k}),\; |S|=m \}$; 
thus, $b_v(m,Q_{2k})$ denotes the minimum
over all node~sets $S\subseteq V(Q_{2k})$ of size~$m$
of the number of neighbors of $S$.
For the sake of exposition,
we mention that 
the node~sets $S$ with $|N_{Q_{2k}}(S)| = b_v(m,Q_{2k})$
(i.e., the minimizers of $b_v(m,Q_{2k})$)
are Hamming balls (see \cite[page~126]{Bol}), and
the formula for $b_v(m,Q_{2k})$ (stated in Theorem~\ref{thm:harper} below)
is obtained by computing the minimum number of neighbors of such sets.
Harper, see \cite{H66} and also see \cite{FF81}, proved the following result:

\begin{theorem}[Theorem~4, Ch.~16, \cite{Bol}] \label{thm:harper}
Every integer $m$, $1\le m \le 2^{2k}-1$, has a unique representation
in the form
\begin{align*}
m &~=~ \sum_{i=r+1}^{2k} {2k \choose i} + m', \quad 0<m' \leq {2k \choose r}, \\
m' &~=~ \sum_{j=s}^r {m_j \choose j}, \quad 1\le s\le m_s<m_{s+1}<\dots<m_r. \\
\tn{Moreover,~} & \\
 b_v(m,Q_{2k}) &~=~ {2k \choose r} - m' + \sum_{j=s}^r {m_j \choose {j-1}}.
\end{align*}
\end{theorem}

\smallskip

\begin{remark}
To find the unique representation of $m$ stated in the above theorem,
we start by taking $r$ to be the largest integer $x\in \{1,\dots,2k\}$
such that $m \leq \sum_{i=x}^{2k} {{2k} \choose i}$, and then we fix
$m'=m-\sum_{i=r+1}^{2k} {{2k} \choose i}$; clearly, $m' \leq {2k \choose r}$.
Then we write $m'$ (uniquely) in the form $\sum_{j=s}^r {m_j \choose j}$;
 for this, we take $m_r$ to be the largest integer $y$ such that
 ${y \choose r} \leq m'$; if $m'= {m_r \choose r}$, then we are done,
otherwise, we iterate by replacing $m'$ and $r$ by $m'-{m_r \choose r}$ and $r-1$, respectively, and then
applying the previous step.
For example, if $k=3$ and $m=17$,
then $r=4$, and $m={6\choose 6}+{6\choose 5}+m'$,
where $m'=10$ and $m'={5 \choose 4}+{4\choose 3}+{2 \choose 2}$.
\end{remark}

In what follows, we use the abbreviated notation $\phi(m)$ for $b_v(m,Q_{2k})$.

Now, our goal is to show that for $m=1,\dots,2^{2k-1}$,
we have $\phi(m) > \min\{ k^2-1,\; (k-1)(m+1) \}$.
This will imply that the hypercube $Q_{2k}$
satisfies the inequalities stated in Theorem~\ref{thm:oriG}.

We first consider the case $m=1,\dots,2k+1$.
We claim that $\phi(m) ~=~ 1 + (m/2) (4k-m-1)$.
This can be easily verified for $m=1$ and $m=2k+1$
 (by applying Theorem~\ref{thm:harper}).
Now, suppose that $m=2,\dots,2k$;
then, observe that the unique representation of
$m$ (see Theorem~\ref{thm:harper}) is $1 + m'$,
where $m'=m-1$ and $r=2k-1$, and moreover,
$m' = {{2k-1} \choose {2k-1}} +  {{2k-2} \choose {2k-2}} +\dots+
	{{2k-m'} \choose {2k-m'}}$,
hence, $\phi(m) ~=~ (2k) - m' + \Big((2k-1)+(2k-2) +\dots+ (2k-m')\Big)
	~=~ 1 + (m/2) (4k-m-1)$.

\begin{fact} \label{f:small-m}
             \label{f:cube-2}
For each $m=1,\dots,2k+1$, we have
\[ \phi(m) ~>~ \min\{ (k-1)(m+1),\; (k-1)(k+1) \}.
\]
\end{fact}
\begin{proof}
{
We have $\phi(m) ~=~  1 + (m/2) (4k-m-1)$, for $m=1,\dots,2k+1$.
Our goal is to show that 
\[ \Delta ~=~ 1 + (m/2) (4k-m-1) -        \min\{ (k-1)(m+1),\; (k-1)(k+1) \}
\]
is positive.

First, suppose that $m \leq k$.
Then, we have
\[
2\Delta ~=~ 2 + m(4k-m-1) - 2(k-1)(m+1) ~=~ m(k-m) + (k+1)(m-2) + 6.
\]
It can be seen that this quantity is $\ge4$ for $1\leq m\leq k$.
\big(For $2\leq m\leq k$, note that $m(k-m)\ge0$ and $(k+1)(m-2)\ge0$,
hence, $2\Delta\ge 6$; moreover, for $m=1$, we have $2\Delta = 4$.\big)

Next, suppose that $k \leq m$.
Then, we have
\[
2\Delta ~=~ 2 + m(4k-m-1) - 2(k-1)(k+1) ~=~ (2k+1-m)(m-k+1) + (m-1)(k-1) + 2.
\]
Clearly, this quantity is $\ge2$ for $1\leq k\leq m\leq 2k+1$.
}
\end{proof}


\begin{fact} \label{f:big-m}
For each $m=2k+2,\dots,2^{2k-1}$, we have
\[ \phi(m) ~>~ (k-1)(k+1).
\]
\end{fact}
\begin{proof}
{
Let $\alpha$ denote
$\sum_{i=0}^{k-1} {{2k} \choose i} = \sum_{i=k+1}^{2k} {{2k} \choose i}$;
observe that
$2^{2k} = \sum_{i=0}^{2k} {{2k} \choose i} = 2\alpha + {{2k} \choose k}$,
hence, $\alpha = \frac12 2^{2k} - \frac12 {{2k} \choose k}$.

Suppose that $m=2^{2k-1}$.
Then $m = \frac12 2^{2k} = \alpha + \frac12 {{2k} \choose k}$, hence,
$\sum_{i=k+1}^{2k} {{2k} \choose i} < m \leq
\sum_{i=k}^{2k} {{2k} \choose i}$.
Hence, for each $m=2k+2,\dots,2^{2k-1}$,
we have $k\leq r \leq 2k-2$
in the unique representation of $m$ given by Theorem~\ref{thm:harper},
i.e., we have
$\ds m = \sum_{i=r+1}^{2k} {2k \choose i} + m', ~\tn{where}~
0 < m' \leq {2k \choose r}, ~\tn{and}~ k\leq r \leq 2k-2$;
moreover, we have
$m' = \sum_{j=s}^r {m_j \choose j}, \quad 1\le s\le m_s<m_{s+1}<\dots<m_r$.
We will use this notation in the rest of the proof.

To complete the proof, we examine two cases, namely,
(1)~$r=k$, and
(2)~$k+1 \le r \le 2k-2$.

\begin{description}
{
\item[Case 1:  $r=k$.]
Since $m = \alpha + m' \le 2^{2k-1}$, we have
$1 \le m' \le 2^{2k-1} - \alpha = \frac12 {2k \choose k}$.
Hence, $\ds \phi(m) =
	{2k \choose r} - m' + \sum_{j=s}^r {m_j \choose {j-1}} \geq
	{2k \choose r} - m' \geq
	{2k \choose r} - \frac12 {2k \choose k} = \frac12 {2k \choose k}$.
Clearly, for $k=3$, we have $\frac12 {2k \choose k} > k^2-1$,
and for $k\ge3$, we have $\frac12 {2k \choose k} \ge \frac12 {2k \choose 3} > k^2-1$.
Moreover, for $k=1$, Fact~\ref{f:big-m} holds vacuously, and
for $k=2$, by the 4-node~connectivity of $Q_4$, we have
$\phi(m) \ge 4 > k^2-1=3,\; \forall m\in\{4,\dots,8\}$.

\item[Case  2:   $k+1 \le r \le 2k-2$.]
Claim~\ref{cl:lower-shadow}, see below, states the key inequality 
\[	m' < \sum_{j=s}^r {m_j \choose {j-1}}.
\]
This immediately implies that
\[ \phi(m) = {2k \choose r} - m' + \sum_{j=s}^r {m_j \choose {j-1}} >
	{2k \choose r} \ge {2k \choose 2} =
	k(2k-1) > k^2-1 ~~(\tn{for~} k\ge1),
\]
as required;
observe that the second inequality uses the upper~bound on $r$
(as well as the lower~bound $r\ge k+1\ge 2$).
}
\end{description}


\medskip

\begin{claim} \label{cl:lower-shadow}
For $r\ge k+1$,
we have
$\ds \sum_{j=s}^r {m_j \choose {j-1}} > m'$.
\end{claim}

To prove this claim, it is convenient to view the $2^{2k}$ nodes
of $Q_{2k}$ as the $2^{2k}$ subsets of the set $\{1,2,\dots,2k\}$
(recall the second construction in Section~\ref{sec:prelims}).

Let $L_i \subset V(Q_{2k})$ denote the set of nodes corresponding to
$i$-element subsets of $\{1,2,\dots,2k\}$.
For $A \subseteq L_i$, let $\Gamma(A)$ denote $N_{Q_{2k}}(A)\cap{L_{i-1}}
~=~ \{ v \in L_{i-1} : \exists w \in A \mbox { such that } \{v,w\}
\in E(Q_{2k}) \}$; $\Gamma(A)$ is called the lower shadow of $A$.
(We mention that the lower shadow of $A$ is denoted by $\partial A$
in \cite{Bol}.)


Following \cite[Ch.5]{Bol}, let $\partial^{(r)}(m')$ denote
$\sum_{j=s}^r {m_j \choose {j-1}}$.

Let $M'\subseteq L_r$ consist of the first $m'$ nodes
(in colex order) of $L_r$, and let
$S' \subseteq L_{r-1}$ consist of the first $\partial^{(r)}(m')$ nodes
(in colex order) of $L_{r-1}$.

It is well known that the lower~shadow of the first $m'$ nodes (in
colex order) of $L_r$ consists of precisely the first
$\partial^{(r)}(m')$ nodes (in colex order) of $L_{r-1}$;
see \cite[pp.~28--32]{Bol}.
Thus, we have $\Gamma(M') = S'$.

Our key inequality can be restated as $\ds m' = |M'| < |S'|$.
We will derive it by examining the subgraph $H$ of $Q_{2k}$ induced by
$M' \cup S'$.
Note that $H$ is a bipartite graph with the node bipartition $M'$, $S'$.
Observe that for each node of $M'$
(which corresponds to an $r$-element set),
there are exactly $r$ neighbors in $\Gamma(M') = S'$.
On the other hand, a node in $S'$
(which corresponds to an $(r-1)$-element set)
has $\leq 2k - r + 1 <  r$ neighbors in $M'$
(the strict inequality follows from $k+1 \leq r$). 
It follows that  $|M'| < |S'|$. 
This proves the inequality 
$\ds \sum_{j=s}^r {m_j \choose {j-1}} > m'$
of our claim. 
}
\end{proof}

\medskip

Our main result follows from Theorem~\ref{thm:oriG}, Theorem~\ref{thm:harper},
the fact that $Q_{2k}$ is $2k$-regular and $2k$-connected,
 and the inequalities stated above
(see Facts~\ref{f:cube-2},~\ref{f:big-m}).

\begin{theorem} \label{thm:main}
{Every} Eulerian orientation
of a hypercube of degree $2k$
is strongly $k$-node connected.
\end{theorem}


\bigskip
\bigskip
\noindent
{\bf Acknowledgments}:
We thank Zoltan Szigeti for several suggestions that
improved the paper, and we thank
Andre Linhares for his comments on a preliminary draft.
We are grateful to other colleagues and reviewers for their comments.






\begin{thebibliography}{99}
{
\bibitem{Bol}
B.Bollob\'{a}s,
\textit{ Combinatorics: Set Systems, Hypergraphs, Families of Vectors
and Combinatorial Probability},
Cambridge University Press, 1986.

\bibitem{BM}
J.A.Bondy and U.S.R.Murty,
\textit{Graph Theory},
Springer, Graduate Texts in Mathematics, Vol.~244, New York, 2008.

\bibitem{CD90}
C.-H.Chou and D.H.C.Du,
``Uni-directional hypercubes,''
Proc.\ Supercomputing '90, (New York, NY, USA, Nov.~12-16, 1990),
254--263, {IEEE} Computer Society, 1990.

\bibitem{Di}
R.Diestel,
\textit{Graph Theory},
Springer, Graduate Texts in Mathematics, Vol.~173 (3rd edition), Berlin~Heidelberg, 2006.


\bibitem{FKL02}
P.Fraigniaud, J.-C.K\"{o}nig, E.Lazard:
``Oriented hypercubes,''
\textit{Networks}, 39(2): 98--106, 2002.

\bibitem{F-HC95}
A.Frank,
``Connectivity and network flows,'' in
\textit{Handbook of Combinatorics} 1:111--177,
Elsevier, Amsterdam, 1995.


\bibitem{FF81}
P.Frankl and Z.F\"{u}redi,
``A short proof for a theorem of Harper about Hamming spheres,''
\textit{Discrete Math.}, 34:311--313, 1981.

\bibitem{H95}
M.Hamdi,
``Topological properties of the directional hypercube,''
\textit{Information Processing Letters}, 53:277--286, 1995.

\bibitem{H66}
L.H.Harper,
``Optimal numberings and isoperimetric problems,''
\textit{J.\ Comb.\ Theory}, 1:385--394, 1966.

\bibitem{KS06}
Z.Kir{\'a}ly and Z.Szigeti,
``Simultaneous well-balanced orientations of graphs,''
\textit{J.\ Comb.\ Theory, Ser.B}, 96(5):684--692, 2006.

\bibitem{KL08}
T.Kir{\'a}ly and L.C.Lau,
``Approximate min-max theorems for
Steiner rooted-orientations of graphs and hypergraphs,''
\textit{J.\ Comb.\ Theory, Ser.B}, 98(6):1233--1252, 2008.


\bibitem{L92}
F.T.Leighton,
\textit{Introduction to Parallel Algorithms and Architectures: Array, Trees, Hypercubes},
Morgan Kaufmann Publishers Inc., San Francisco, 1992.


\bibitem{N60}
C.St.J.A.Nash--Williams,
``On orientations, connectivity, and odd vertex pairings in finite graphs,''
\textit{Canad.\ J.\ Math.}, 12:555--567, 1960.

\bibitem{wiki-HG}
Hypercube graph,
article in Wikipedia.
URL: https://en.wikipedia.org/wiki/Hypercube\_graph

\bibitem{SS88}
Y.Saad and M.H.Schultz,
``Topological properties of hypercubes,''
\textit{IEEE Transactions on Computers}, 37:867--872, 1988.

\bibitem{T89}
C.Thomassen,
``Configurations in graphs of large minimum degree,
connectivity, or chromatic number,''
\textit{Ann. NY Acad. Sci.}, 555:402--412, 1989.

\bibitem{T2015}
C.Thomassen,
``{Strongly 2-connected orientations of graphs},''
\textit{J.\ Comb.\ Theory, Ser.B}, 110:67--78, 2015.
}
\end{thebibliography}
\end{document}